\documentclass[12pt, reqno]{amsart}
\usepackage{amsmath, amsthm, amscd, amsfonts, amssymb, graphicx, xcolor}
\usepackage[bookmarksnumbered, colorlinks, plainpages]{hyperref}

\newcommand{\U}{\mathbin{\mathcal{U}\kern-.1em}}
\renewcommand{\S}{\mathbin{\mathcal{S}\kern-.08em}}

\newcommand{\A}{\mathbin{A}}

\renewcommand{\a}{\alpha}
\renewcommand{\b}{\beta}

\newcommand{\mc}{\mathcal}

\newtheorem{theorem}{Theorem}[section]
\newtheorem{lemma}[theorem]{Lemma}

\newtheorem{corollary}[theorem]{Corollary}
\theoremstyle{definition}
\newtheorem{definition}[theorem]{Definition}

\theoremstyle{remark}

\numberwithin{equation}{section}

\begin{document}
\setcounter{page}{1}


\centerline{}

\centerline{}


\title[Semi-Symmetric Non-Metric Connection]{K-contact manifolds admitting some geometric solitons with Semi-Symmetric Non-Metric Connection}

\author[Mondal, Basu, Bhattacharyya]{Bidhan Mondal$^1$,  Nirabhra Basu$^2$$^{*}$, Arindam Bhattacharyya$^3$$^*$}

\address{$^{1}$ Department of Mathematics, Jadavpur University, Kolkata, India.}
\email{\textcolor[rgb]{0.00,0.00,0.84}{bidhanmondal381@gmail.com}}

\address{$^{2}$ Department of Mathematics, The Bhawanipur Education Society College, Kolkata, India.}
\email{\textcolor[rgb]{0.00,0.00,0.84}{basu.nirabhra@gmail.com}}

\address{$^{3}$ Department of Mathematics, Jadavpur University, Kolkata, India.}
\email{\textcolor[rgb]{0.00,0.00,0.84}{arindambhat16@gmail.com}}


\date{
\newline \indent $^{*}$ Corresponding author
}

\begin{abstract}
In this paper, we introduce some type vector fields with respect to a semi-symmetric non-metric (SSNM) connection. We investigate several geometric properties of a K-contact manifold equipped with an SSNM connection and provide a concrete example to justify the relation between the scalar curvature of the SSNM connection and Levi-Civita connection that we have obtained in this paper. Furthermore, we have found the nature of Riemann solitons, conformal Ricci solitons and conformal $\eta$-Ricci-Yamabe solitons on K-contact manifolds admitting a SSNM connection.\\
Finally, we determine the necessary and sufficient conditions for such a manifold  to be $\Tilde{\tau}$-semi-symmetric, quasi-conformal-semi-symmetric and pseudo-projective-semi-symmetric.
\end{abstract} \maketitle


\section{Introduction:}\label{sec1}

In differential geometry, connection is a very important differential operator which connects to different points in tangent space on a manifold. In 1917, the mathematician Levi-Civita introduced a special type of differential operator, named as Levi-Civita connection  $(\nabla)$ \cite{MR144278}, where torsion tensor $(\Tilde{T})$ and the covariant derivative of the metric tensor $(g)$ vanish identically i.e $\Tilde{T}=0$ and $\nabla g=0$. The concept of a semi-symmetric connection on a differential manifold was introduced by A. Friedmann and J. A. Schouten in 1924 \cite{MR1544701}.  In 1932, Hayden \cite{MR1576150} introduced same idea with torsion on a Riemannian manifold. In a Riemannian manifold $(M,g)$ consider a linear connection $\Tilde{\nabla}$, which is said to be semi-symmetric if the torsion tensor $\Tilde{T}$ satisfies\\
$$\Tilde{T}(X,Y)=\omega(Y)X-\omega(X)Y,$$
for all vector field $X,Y$ on $M$ where $\omega$ is a 1-form associated with a non zero vector field $J$ defined by $g(X,J)=\omega(X).$ M. M. Tripathi \cite{MR1768254}  studied semi-symmetric metric connection in a Kenmotsu manifold. 
 In 1992, Agashe and Chafle \cite{MR1170219} introduced the semi-symmetric non-metric connection on a Riemannian manifold. 
An $n$-dimensional differentiable manifold $(M,g)$ of class $\mathbf{C^\infty}$ equipped with a linear connection $\Tilde{\nabla}$ is said to be a semi-symmetric non-metric connection(SSNM), if its torsion tensor $\Tilde{T}$ satisfies:
\begin{align}\label{eq1}
     \Tilde{T}(X,Y)=\omega(Y)X-\omega(X)Y,
\end{align}
\begin{align}\label{eq2}
    (\Tilde{\nabla}_Xg)(Y,Z)=-\omega(Y)g(X,Z)-\omega(Z)g(X,Y),
\end{align}
for all vector fields $X,Y,Z$ on $M$.\\
Later several geometers such as S. K. Chaubey, U. C. De, M. M. Tripathi, N. Nakkar have studied semi-symmetric non-metric connection and established various important properties.\\
Hamilton introduced the concept of Ricci flow \cite{MR2061425} in 1982 to prove Thurston's geometric conjecture. Ricci flow has became one of the most important concepts for the studying of differential geometry. The Ricci flow equation is given by:
$$\frac{\partial}{\partial t}g=-2S,$$
where $S$ is the Ricci curvature tensor and $g(t)$ is a one-parameter family of metrics on M. In 1988, Hamilton \cite{MR2274812} introduced the concept of Ricci soliton as well. A Riemannian manifold $(M,g)$ is said to admit a Ricci soliton if there exists $V$ satisfying  
\begin{equation*}
     \frac{1}{2}\mc{L}_{V}g+S(g)=\lambda g,
 \end{equation*}
 where  $\mc{L}_V$ is the Lie derivative along the direction of the vector field $V$ and $\lambda\in\mathbb{R}$. The soliton is called expanding, steady or shrinking if $\lambda<0,=0$ or $>0$ respectively.\\
 In 2004, E. Fischer \cite{MR2053005} introduced a new idea which modifies the unit volume constraint of the Ricci flow equation to a scalar curvature constraint, known as the conformal Ricci flow equation.
 This equation in $(M,g)$ is given by:
        \begin{equation*}
            \frac{\partial}{\partial t}g+2(S+\frac{1}{n}g)=-pg, \quad r=-1,
        \end{equation*}
where $r$ is the scalar curvature and $p$ is a time-dependent scalar field called as conformal pressure which serves as a Lagrange multiplier to conformally deform the metric flow so as maintain the scalar curvature constrain.
\begin{definition}
    A Riemannian manifold $(M,g)$ is said to admit a conformal Ricci soliton \cite{MR3343178} if it satisfies \\
    \begin{equation}\label{eq3}
        \mc{L_V}g+2S=[2\lambda-(p+\frac{2}{n})]g.
    \end{equation}
    This concept was introduced by N. Basu and A. Bhattacharyya in 2015. They also showed its existence \cite{MR4728399}.
\end{definition}
The concept of a conformal $\eta$-Ricci-Yamabe soliton was introduced by P. Zhang, Y. Li, S. Roy and A. Bhattacharyya \cite{MR4404746} in 2022 as a generalization merging conformal geometry, Ricci-Yamabe flow and $\eta$-solitons.
\begin{definition}
    An $n$-dimensional Riemannian manifold $(M,g)$ is said to admit conformal $\eta$-Ricci-Yamabe soliton if it satisfies
    \begin{align}\label{eq4}
        \mc{L_V}g+2\a S+[2\lambda-\b r-(p+\frac{2}{n})]g+2\gamma\eta\otimes\eta=0,
    \end{align}
    where $\gamma$ is a constant, $\a$ and $\b$ are real scalar. If $\gamma=0$, the soliton structure reduces to a conformal Ricci-Yamabe soliton. also the soliton is said to be Ricci-Yamabe soliton if $\gamma=0$ and $p$ tends to $0$.  The soliton is called expanding, steady or shrinking if $\lambda>0,~\lambda=0$ or $\lambda<0$ respectively.
\end{definition}
Constantin Udriste introduced Riemann flow \cite{MR2849340} in 2010 which is generalized concept of Ricci flow and the flow defined as \\
$$\frac{\partial}{\partial t}G(t)=-2R(g(t)),$$
where $R$ is the Riemann curvature and $G=\frac{1}{2}g\otimes g$. Here the symbol $\otimes$ denotes the Kulkarni-Nomizu product defined by
\begin{align*}
    (g\otimes h)(X,Y,Z,W)=g(X,W)h(Y,Z)+g(Y,Z)h(X,W)-g(X,Z)h(Y,W)\\-g(Y,W)h(X,Z),\quad\quad\quad\quad\quad\quad\quad\quad\quad\quad\quad\quad\quad\quad\quad\quad
\end{align*}
where $g,~h$ both are Riemannian metrics.
A Riemannian manifold $(M,g)$ admits a Riemann soliton if there exists a vector field $V$ satisfying
\begin{equation}\label{eq5}
    \frac{1}{2}g\otimes\mc{L_V}g+R=\lambda G.
\end{equation}\\
Several authors studied the geometrical properties of SSNM connection in different context. In 2016, T. Demirli, C. Ekici and A. Gorgulu have given some results on Ricci solitons with SSNM conection on f-Kenmotsu manifolds \cite{DTC}. The authors G. Ingalahalli and C. S. Bagewadi introduced several curvature tensor on K-contact manifold with SSNM conection \cite{IS} in 2017.\\
Inspired by these investigations, in this paper, we organized as follows. In section 2, we recall some fundamental formula on K-contact manifold. In section 3, we have highlighted some relation with SSNM connection on $K$-contact manifold and given an example.  In Section 4, we establish key properties of $K$-contact manifolds equipped with various types of solitons whose potential vector field $V$ is torse-forming and holomorphically planar conformal. Finally we study $\tau-$curvature on K-contact manifold with respect to SSNM connection.

 \section{Preliminaries:}\label{sec2}

A $(2n+1)-$dimensional manifold $M$ has an almost contact structure if it admits a vector field $\xi$, a tensor field $\varphi$ of type $(1,1)$ and a 1-form $\eta$ satisfying
\begin{equation}\label{eq6}
    \varphi^2=-I+\eta\otimes\xi,~~\eta(\xi)=1.
\end{equation}
Almost contact structure was defined by J. Gray \cite{MR467588} in 1959 . From above equations, we can easily derived that
\begin{equation}\label{eq7}
    \varphi\cdot\xi=0,~~~ \eta\cdot\varphi=0.
\end{equation}
If the manifold $M$ with almost contact structure $(\varphi,\xi,\eta)$ admits a Riemannian metric $g$ such that 
\begin{equation*}
    g(\varphi X,\varphi Y)=g(X,Y)-\eta(X)\eta(Y),
\end{equation*}
for all $X,~Y\in TM$, then the structure $(\varphi,\xi,\eta,g)$ is called almost contact metric structure. 
An almost contact metric structure $(\varphi,\xi,\eta,g)$ is said to be contact metric structure on $M$, if the 1-form $\eta$ satisfies
\begin{equation*}
    d\eta(X,Y)=g(X,\varphi Y).
\end{equation*}
 A contact metric manifold $(M,g)$ is called K-contact manifold if the vector field $\xi$ is Killing i.e, $\mc{L_\xi}g=0$. In a K-contact manifold $(M,g)$, the following relations hold:
 \begin{equation}\label{eq8}
     (\nabla_X\eta)Y=g(\nabla_X\xi,Y),~~~ \nabla_X\xi=-\varphi X,
 \end{equation}
\begin{equation}\label{eq9}
    R(\xi,X)Y=g(X,Y)\xi-\eta(Y)X,
\end{equation}
\begin{equation}\label{eq10}
     R(X,Y)\xi=\eta(Y)X-\eta(X),
\end{equation}
\begin{equation}\label{eq11}
    g(R(X,Y)Z,\xi)=g(Y,Z)\eta(X)-g(X,Z)\eta(Y),
\end{equation}
\begin{equation}\label{eq12}
    S(X,\xi)=2n\eta(X),
\end{equation}
\begin{equation}\label{eq13}
    Q\xi=2n\xi,
\end{equation}
for any vector fields $X,Y,Z$ in $M$. Here $Q$ is the Ricci operator of $M$.
If the Ricci tensor S of an $(2n+1)-$dimensional manifold $M$ is parallel, then $M$ is an Einstein manifold \cite{MR2934086} and $S(X,Y)=2ng(X,Y),~~ QX=2nX$.\\

\section{Semi-Symmetric Non-Metric Connection:}\label{sec3}
A $(2n+1)$-dimensional Riemannian manifold $(M,g)$ admits a SSNM connection $\Tilde{\nabla}$ such that \cite{MR1170219}
\begin{equation}\nonumber
    \Tilde{\nabla}_XY=\nabla_XY+\omega(Y)X,
\end{equation}
where $\nabla$ is the Levi-Civita connection. A relation between Riemannian curvature tensor $\Tilde{R}$ and $R$ with respect to SSNM connection $\Tilde{\nabla}$ and Levi-Civita connection $\nabla$ on a Riemannian manifold $M$ is given by
\begin{equation}\nonumber
    \Tilde{R}(X,Y)Z=R(X,Y)Z+((\nabla_X\omega)Z)Y-\omega(Z)\omega(X)Y-((\nabla_Y\omega)Z)X+\omega(Z)\omega(Y)X,
\end{equation}
and the relation between Ricci tensor $\Tilde{S}$ and $S$ with same connection is
\begin{equation}\nonumber
    \Tilde{S}(X,Y)=S(X,Y)-2n\a(X,Y),
\end{equation}
where $\a$ is a $(0,2)$ tensor field, defined by $\a(X,Y)=(\nabla_X\omega)Y-\omega(Y)\omega(X)$.\\
If we consider the 1-form $\omega$ to be $\eta$, then we have the expression for the Riemannian curvature tensor on K-contact manifold with SSNM connection as
\begin{equation}\label{eq14}
    \Tilde{R}(X,Y)Z=R(X,Y)Z-g(Z,\varphi X)Y+g(Z,\varphi Y)X-\eta(Z)\eta(X)Y+\eta(Z)\eta(Y)X.
\end{equation}
Using (\ref{eq9}),(\ref{eq10}),(\ref{eq11}) and (\ref{eq13}) in above equation, we find the following results
\begin{equation}\label{eq15}
    \Tilde{R}(X,Y)\xi=2(\eta(Y)X-\eta(X)Y),
\end{equation}
\begin{equation}\label{eq16}
    \Tilde{R}(\xi,Y)Z=g(Y,Z)\xi+g(\varphi Y,Z)\xi-2\eta(Z)Y+\eta(Z)\eta(Y)\xi,
\end{equation}
\begin{equation}\label{eq17}
    \Tilde{S}(X,Y)=S(X,Y)-2ng(X,\varphi Y)+2n\eta(X)\eta(Y),
\end{equation}
\begin{equation}\label{eq18}
    \Tilde{S}(X,\xi)=4n\eta(X),~~ \Tilde{Q}\xi=4n\xi,
\end{equation}
\begin{equation}\label{eq19}
    \Tilde{r}=r+2n,
\end{equation}
 where $\Tilde{Q}$ is the Ricci operator with SSNM connection of $M$.

\subsection{Example:}
In this section, we construct an example of 3-dimensional K-contact manifold $M=\{(x,y,z)\in \mathbf{R^3}:z\neq0\}$, where $(x,y,z)$ are standard coordinates in $\mathbf{R^3}$. The vector field defined as
\begin{align*}
   e_1=\frac{\partial}{\partial x},~~e_2=2\frac{\partial}{\partial y},~~e_3=\frac{\partial}{\partial z}-y\frac{\partial}{\partial x}, 
\end{align*}
are linearly independent at each point of $M$. Let $e_1=\xi$ and $g$ be the Riemannian metric defined by
\begin{align*}
    g(e_i,e_i)=1~ and~ g(e_i,e_j)\neq 0,~where~i\neq j\in 1,2,3.
\end{align*}
Now we take a $1$-form $\eta$ corresponding to the metric g, which is given by $\eta(X)=g(X,\xi)$ for any $X\in\mathfrak{X}(M)$ and a $(1,1)$ tensor field $\varphi$ is defined by 
\begin{align*}
     \varphi(e_1)=0, \quad \varphi(e_2)=-e_3, \quad \varphi(e_3)=e_2.
\end{align*}
From the above equation, we can easily check all the  conditions of almost contact manifold for all $e_i,~i=1,2,3$. Using the definition of Lie bracket $[X,Y]f=X(Yf)-Y(Xf)$, we calculate all Lie bracket follows as: 
\begin{align*}
    [e_1,e_2]=0,~~ [e_2,e_3]=-2e_1,~~[e_1,e_3]=0.
\end{align*}
Let $\nabla$ be a Levi-Civita connection with respect to Riemannian metric $g$. Using the Koszul formula $2g(\nabla_XY,Z)=Xg(Y,Z)+Yg(X,Z)-Zg(Y,X)-g(X,[Y,Z])-g(Y,[X,Z])+g(Z,[X,Y])$, we get
\begin{align*}
    \nabla_{e_1}e_1=0,~~~~~\nabla_{e_1}e_2=e_3,~~~~\nabla_{e_1}e_3=-e_2,\\
\nabla_{e_2}e_1=e_3,~~~~~\nabla_{e_2}e_2=0,~~~~\nabla_{e_2}e_3=-e_1,\\
\nabla_{e_3}e_1=-e_2,~~~~~\nabla_{e_3}e_2=e_1,~~~~\nabla_{e_3}e_3=0.\\
\end{align*}
We take a vector field $W=w^1e_1+w^2e_2+w^3e_3$ where $w^1,w^2,w^3$ are scalars. This vector field satisfies equation (\ref{eq8}). Now we can say that the structure $(\varphi,\xi,\eta,g)$ is a K-contact structure on $M$.\\
We chose a non-metric connection $\Tilde{\nabla}$ such that $\Tilde{\nabla}_XY=\nabla_XY+\eta(Y)X,~~\forall~X,Y\in\mathfrak{X}(M)$, we have
\begin{align*}      \Tilde{\nabla}_{e_1}e_1=e_1,~~\Tilde{\nabla}_{e_1}e_2=e_3,~~\Tilde{\nabla}_{e_1}e_3=-e_2,\\
    \Tilde{\nabla}_{e_2}e_1=e_3+e_2,~~\Tilde{\nabla}_{e_2}e_2=0,~~\Tilde{\nabla}_{e_2}e_3=-e_1,\\
    \Tilde{\nabla}_{e_3}e_1=-e_2+e_3,~~\Tilde{\nabla}_{e_3}e_2=e_1,~~\Tilde{\nabla}_{e_3}e_3=0.
\end{align*}
We verify all conditions of the semi-symmetric non-metric connection, i.e, $\Tilde{T}(X,Y)=\eta(Y)X-\eta(X)Y$ and $
\Tilde{\nabla}_Xg(Y,Z)=-\eta(Y)g(X,Z)-\eta(Z)g(X,Y)$. After that we calculate all possible value of Riemannian curvature tensor $(R~ and~ \Tilde{R})$ and Ricci tensor $(S~and~\Tilde{S})$ with respect to the connection of $\nabla$ and $\Tilde{\nabla}$, we have,
\begin{align*}
    R(e_1,e_2)e_1=-e_2,~~R(e_1,e_2)e_2=e_1,~~R(e_1,e_2)e_3=0,\\
    R(e_1,e_3)e_1=-e_3,~~R(e_1,e_3)e_2=0,~~R(e_1,e_3)e_3=e_1,\\
    R(e_2,e_3)e_1=0,~~R(e_2,e_3)e_2=3e_3,~~R(e_2,e_3)e_3=-3e_2,
\end{align*}
and
\begin{align*}
    \Tilde{R}(e_1,e_2)e_1=-2e_2,~~\Tilde{R}(e_1,e_2)e_2=e_1,~~\Tilde{R}(e_1,e_2)e_3=-e_1,\\
    \Tilde{R}(e_1,e_3)e_1=-2e_3,~~\Tilde{R}(e_1,e_3)e_2=e_1,~~\Tilde{R}(e_1,e_3)e_3=e_1,\\
    \Tilde{R}(e_2,e_3)e_1=0,~~\Tilde{R}(e_2,e_3)e_2=3e_3+e_2,~~\Tilde{R}(e_2,e_3)e_3=-3e_2+e_3,
\end{align*}
and
\begin{align*}
    S(e_1,e_1)=2,~~S(e_2,e_2)=-2,~~S(e_3,e_3)=-2,
\end{align*}
and
\begin{align*}
    \Tilde{S}(e_1,e_1)=4,~~\Tilde{S}(e_2,e_2)=-2,~~\Tilde{S}(e_3,e_3)=-2.
\end{align*}
This show that the scalar curvature $r=-2$ and $\Tilde{r}=0$ with respect to $\nabla$ and $\Tilde{\nabla}$ which satisfies the relation $\Tilde{r}=r+2n$.

\section{Some types of solitons with SSNM connection on a K-contact manifold:}\label{sec4}
In this section we discuss some type of vector fields with SSNM connection. We determine the nature of solitons on SSNM connection and find the value of the scalar curvature on metric connection and SSNM connection on 
$K$-contact manifold. After that we discuss the characteristics on the conformal Ricci Soliton with respect to the potential vector field $\xi$ and conformal $\eta$-Ricci-Yamabe soliton with respect to torse-forming vector field $V$ on $K$-contact manifold with SSNM connection.\\
 The torse-forming vector fields $V$, which was introduced by K.Yano \cite{MR14777} in 1944, this vector fields are highly significant in the study of differential geometry and physics, particularly within the theory of submanifolds.
\begin{definition}
    A vector field $V$ is said to be a torse-forming vector field on a Riemannian manifold $(M,g)$ if it satisfies the following:
    \begin{align*}
        \nabla_XV=\mu X+\psi(X)V,
    \end{align*}
    where $\nabla$ denotes the Levi-Civita connection, $X$ is any vector field on $M$, $\mu$ is a smooth function and $\psi$ is a 1-form known as the generating form.
\end{definition}
Now we express the torse-forming vector fields with SSNM connection on $M$ given as:
\begin{align}\nonumber
    \Tilde{\nabla}_XV=\mu X+\psi(X)V+\omega(V)X,
\end{align}
where $\Tilde{\nabla}$ denotes the SSNM connection and $\omega$ is another 1-form. If the value of the 1-form $\psi$ is zero, then the torse-forming vector field $V$ reduces to concircular vector field \cite{CBY}. If $\psi$ vanishes identically and $\mu=1$, we call the vector field $V$ is concurrent vector field \cite{MR296863}. If $\mu=\psi=0$, then $V$ is parallel vector field. Also the vector field $V$ is called a recurrent vector field if the value of $\mu$ is zero.\\
In 2003, R. Sharma \cite{MR2457028} defined the notion of a holomorphically planar conformal vector (HPCV) field on an almost Hermitian manifold. Studied the paper \cite{MR2457028}, we extend this concept with semi-symmetric non-metric (SSNM) connection.
\begin{definition}
    A vector field $V$ on a Reimannian manifold $(M,g)$ with a SSNM connection $\Tilde{\nabla}$, is called a holomorphically planar conformal vector (HPCV) field if it satisfies the following equation
    \begin{align}\label{eq20}
        \Tilde{\nabla}_XV=\mu X+\theta\varphi X +\omega(V)X,
    \end{align}
    where $\mu$ and $\theta$ are smooth functions and $\varphi$ is a (1,1) tensor field on $M$.
\end{definition}

\begin{lemma}
   If $V$ is a potential vector field, then the Lie derivative of the Riemannian metric $g$ along with $V$ with respect to the SSNM connection is $(\Tilde{\mc{L}}_Vg)(X,Y)=g(\Tilde{\nabla}_XV,Y)+g(X,\Tilde{\nabla}_YV)-2\omega(V)g(X,Y).$
\end{lemma}

\begin{proof}
    Let $g$ be a Riemannian metric, then by the formula for Lie derivative of $g$ along a vector field $V$ with SSNM connection as follow:
    \begin{align*}
        (\Tilde{\mc{L}}_Vg)(X,Y)=Vg(X,Y)-g(\Tilde{\mc{L}}_VX,Y)-g(X,\Tilde{\mc{L}}_VY),
    \end{align*}
    for all vector field $X,Y$ on $M$.\\
    Using the formula of Lie derivative of a vector field along $V$, we get
    \begin{align*}
        (\Tilde{\mc{L}}_Vg)(X,Y)=Vg(X,Y)-g(\Tilde{\nabla}_VX,Y)+g(\Tilde{\nabla}_XV,Y)+g(\Tilde{T}(V,X),Y)\\-g(X,\Tilde{\nabla}_VY)+g(X,\Tilde{\nabla}_YV)+g(X,\Tilde{T}(V,Y)).~~~~~~~~~~~~~~~~~~~~~~~~~~~~~~~~~~~~~~~~~~~~~~~~
    \end{align*}
    Substituting the value of the torson tensor $\Tilde{T}$ from (\ref{eq1}) into above equation, we obtain
    \begin{align}\label{eq21}
        (\Tilde{\mc{L}}_Vg)(X,Y)=g(\Tilde{\nabla}_XV,Y)+g(X,\Tilde{\nabla}_YV)-2\omega(V)g(X,Y).
    \end{align}
If we take the vector field $V$ is a torse-forming vector field, then above equation becomes
\begin{align}\label{eq22}
 (\Tilde{\mc{L}}_Vg)(X,Y)=2\mu g(X,Y)+\psi(X)g(V,Y)+\psi(Y)g(V,X).   
\end{align}
Similarly taking the vector field $V$ as a holomorphically planar conformal vector (HPCV) field, we get
\begin{align}\label{eq23}
    (\Tilde{\mc{L}}_Vg)(X,Y)=2\mu g(X,Y).
\end{align}
\end{proof}

\begin{theorem}
     If a $(2n+1)-$dimensional $K$-contact manifold $(M,\varphi,\xi,\eta,g)$ with SSNM connection admits a Riemann soliton whose potential vector field $V$ is a torse-forming vector field, then $\lambda=\frac{1}{2n}[\frac{\Tilde{r}-2\psi(V)}{2n+1}+4n\mu]$ and the soliton is shrinking if $\Tilde{r}>2\psi(V)-4n(2n+1)\mu$, steady if $\Tilde{r}=2\psi(V)-4n(2n+1)\mu$ or expanding if $\Tilde{r}<2\psi(V)-4n(2n+1)\mu$. 
\end{theorem}
\begin{proof}
    Let $(M,g)$ be an $(2n+1)$ dimensional $K$-contact manifold. Using the product rule of Kulkarni-Nomizu, the equation (\ref{eq5}) becomes 
    \begin{align*}
        2\Tilde{R}(X,Y,Z,W)+g(X,W)(\Tilde{\mc{L}}_Vg)(Y,Z)+g(Y,Z)(\Tilde{\mc{L}}_Vg)(X,W)-g(X,Z)(\Tilde{\mc{L}}_Vg)(Y,W)\\-g(Y,W)(\Tilde{\mc{L}}_Vg)(X,Z)=2\lambda[g(X,W)g(Y,Z)-g(X,Z)g(Y,W)].~~~~~~~~~~~~~~
    \end{align*}
Contracting over $X$ and $W$, we get
\begin{align*}
    2\Tilde{S}(Y,Z)+(2n-1)(\Tilde{\mc{L}}_Vg)(Y,Z)+2g(Y,Z)div(V)=4n\lambda g(Y,Z).~~~~~~~~~~~~~~
\end{align*}
We take the vector field $V$ as a torse-forming vector field. Now using (\ref{eq22}) in the above equation, we obtain
\begin{align}\label{eq24}
    \Tilde{S}(Y,Z)=\{2n\lambda-4\mu n-\psi(V)\} g(Y,Z)-\frac{(2n-1)}{2}[\psi(Y)g(V,Z)+\psi(Z)g(V,Y)].
\end{align}
By contracting (\ref{eq24}) once more, we obtain,
\begin{align*}
    \lambda=\frac{1}{2n}[\frac{\Tilde{r}-2\psi(V)}{2n+1}+4n\mu].
\end{align*}
Consequently, the soliton is classified as shrinking, steady, or expanding depending on whether $\lambda>0,~\lambda=0,$ or $\lambda<0,$ respectively. This completes the proof.
\end{proof}
Depending on the value of $\mu$ and $\psi$, we can conclude the following:
\begin{corollary}
    Let a $(2n+1)-$dimensional $K$-contact manifold $(M,g)$ admits a Riemann soliton with respect to SSNM connection where potential vector field $V$ is a torse-forming vector field. If $V$ is a:
    \begin{itemize}
        \item concircular vector field, then $\lambda=[\frac{\Tilde{r}}{2n(2n+1)}+2\mu]$ and the soliton is shrinking if $\Tilde{r}>-4n(2n+1)\mu$, steady if $\Tilde{r}=-4n(2n+1)\mu$ or, expanding if $\Tilde{r}<-4n(2n+1)\mu$;
        \item concurrent vector field, then $\lambda=\frac{1}{2n}[\frac{\Tilde{r}}{2n+1}+4n]$ and the soliton is shrinking if $\Tilde{r}>-4n(2n+1)$, steady if $\Tilde{r}=-4n(2n+1)$ or, expanding if $\Tilde{r}<-4n(2n+1)$;
        \item recurrent vector field, then $\lambda=[\frac{\Tilde{r}-2\psi(V)}{2n(2n+1)}]$ and the soliton is shrinking if $\Tilde{r}>2\psi(V)$, steady if $\Tilde{r}=2\psi(V)$ or, expanding if $\Tilde{r}<2\psi(V)$.
    \end{itemize}
\end{corollary}

\begin{corollary}
    If a $(2n+1)-$dimensional $K$-contact manifold $(M,\varphi,\xi,\eta,g)$ with SSNM connection admits a Riemann soliton whose potential vector field $V$ is a HPCV field, then $\lambda=\frac{\Tilde{r}}{n(2n+1)}+(2+\frac{1}{n})\mu$ and the soliton is shrinking if $\Tilde{r}>-(2n+1)^2\mu$, steady if $\Tilde{r}=-(2n+1)^2\mu$ or, expanding if $\Tilde{r}<-(2n+1)^2\mu$.
\end{corollary}
\begin{proof}
    Let $(M,\varphi,\xi,\eta,g)$ be an $(2n+1)-$dimensional  $K$-contact manifold with SSNM connection admits Riemann soliton. If $V$ is a HPCV field, then using (\ref{eq23}) on the equation of Riemann soliton, we get
    \begin{align}\label{eq25}
        \Tilde{S}(X,Y)=[n\lambda-(2n+1)\mu]g(X,Y).
    \end{align}
    So the manifold $M$ is a Einstein manifold.
    Now, contracting the above equation (\ref{eq25}) for $X$ and $Y$, we get
    \begin{align}\nonumber
        \lambda=\frac{\Tilde{r}}{n(2n+1)}+(2+\frac{1}{n})\mu.
    \end{align}
    
\end{proof}

\begin{theorem}
    A $(2n+1)-$dimensional $K-$contact manifold with SSNM connection admitting conformal Ricci soliton and if the potential vector field $V$ is point wise collinear with $\xi$, then $V$ is constant multiple of $\xi$ and the value of scalar curvature $\Tilde{r}=2n[\lambda-f-(\frac{p}{2}+\frac{1}{2n+1})]$.
    \end{theorem}
    \begin{proof}
        If we replace $V$ by $f\xi$ where $f$ is a smooth function on conformal Ricci soliton equation $(\ref{eq3})$, we get
$$\Tilde{\mc{L}}_{f\xi}g(X,Y)+2\Tilde{S}(X,Y)=[2\lambda-(p+\frac{2}{2n+1})]g(X,Y).$$
Using (\ref{eq21})
\begin{equation*}
    X(f)\eta(Y)+Y(f)\eta(X)-2f[\eta(X)\eta(Y)-g(X,Y)]+2\Tilde{S}(X,Y)=[2\lambda-(p+\frac{2}{2n+1})]g(X,Y)
\end{equation*}
Replacing $Y$ by $\xi$, we obtain,
\begin{equation}\label{eq26}
   X(f)\eta(\xi)+\xi(f)\eta(X)+2\Tilde{S}(X,\xi)=[2\lambda-(p+\frac{2}{2n+1})]g(X,\xi).
\end{equation}
Again replacing $X$ by $\xi$, we get
$$f(\xi)=\lambda-(\frac{p}{2}+\frac{1}{2n+1})-4n.$$
Now we put the value of $f(\xi)$ in (\ref{eq26}), we get
$$X(f)=[\lambda-4n-(\frac{p}{2}+\frac{1}{2n+1})]\eta(X)$$
\begin{equation}\nonumber
   df=[\lambda-4n-(\frac{p}{2}+\frac{1}{2n+1})]\eta(X).
\end{equation}
From the Poincare lemma, we know that $d^2=0$. Applying exterior differentiation operator $d$ in the above equation, we get
$$[\lambda-4n-(\frac{p}{2}+\frac{1}{2n+1})]d\eta=0.$$
Since $d\eta\neq 0.$ So we have 
\begin{equation}\label{eq27}
    \lambda=4n+\frac{p}{2}+\frac{1}{2n+1}.
\end{equation}
 From (\ref{eq27}), we obtain $df=0$, which implies that 
 $f=constant.$\\
 Replacing $X,Y$ by $e_i$, where ${e_i}$ is an orthonormal basis of the tangent space
at each point of the manifold and taking summation over $i=1,2,.....2n+1$ we get,
$$\Tilde{r}=2n[\lambda-f-(\frac{p}{2}+\frac{1}{2n+1})].$$
This completes the proof.
    \end{proof}

\begin{theorem}
     A $(2n+1)-$dimensional $K$-contact manifold $(M,\varphi,\xi,\eta,g)$ with SSNM connection admits a conformal $\eta$-Ricci-Yamabe soliton whose potential vector field $V$ is a torse-forming, then $\lambda=\gamma-\mu-\psi(\xi)\eta(V)-4n\a+\frac{1}{2}[\b(r+2n)+(p+\frac{2}{2n+1})]$ and the soliton is shrinking if $\gamma+\frac{1}{2}[\b(r+2n)+(p+\frac{2}{2n+1})]<\mu+\psi(\xi)\eta(V)+4n\a$, steady if $\gamma+\frac{1}{2}[\b(r+2n)+(p+\frac{2}{2n+1})]=\mu+\psi(\xi)\eta(V)+4n\a$ or, expanding if $\gamma+\frac{1}{2}[\b(r+2n)+(p+\frac{2}{2n+1})]>\mu+\psi(\xi)\eta(V)+4n\a$. 
\end{theorem}
\begin{proof}
    We consider a $K$-contact manifold $(M^{2n+1},g)$ admits conformal $\eta$-Ricci-Yamabe soliton with SSNM connection. The equation $(\ref{eq4})$ becomes:
    \begin{align*}
        (\Tilde{\mc{L}}_Vg)(X,Y)+2\a\Tilde{S}(X,Y)+[2\lambda-\b\Tilde{r}-(p+\frac{2}{n})]g(X,Y)+2\gamma\eta (X)\eta(Y)=0,
    \end{align*}
    where $X,~Y$ on $M$.\\
    Assuming that the potential vector field $V$ is torse-forming, substituting equation (\ref{eq17}), (\ref{eq19}) and (\ref{eq21}) into above expression yields:
    \begin{align}\label{eq28}
        2\a S(X,Y)=4n\a g(X,\varphi Y)-[2\lambda+2\mu-\b(r+2n)-(p+\frac{2}{2n+1}]g(X,Y)\nonumber\\+(2\gamma-4n\a)\eta(X)\eta(Y)-\psi(X)g(V,Y)-\psi(Y)g(V,X).
    \end{align}
    Replacing $X$ and $Y$ by $\xi$ and utilizing equations (\ref{eq6}), (\ref{eq7}) and (\ref{eq12}) in (\ref{eq28}), we obtain:
    \begin{align*}
        \lambda=\gamma-\mu-\psi(\xi)\eta(V)-4n\a+\frac{1}{2}[\b(r+2n)+(p+\frac{2}{2n+1})].
    \end{align*}
    Consequently, the soliton is classified as shrinking, steady or expanding depending on whether $\lambda<0,~\lambda=0$ or $\lambda>0,$ respectively. This completes the proof.
\end{proof}

\begin{corollary}
    If a $(2n+1)-$dimensional $K$-contact manifold $(M,\varphi,\xi,\eta,g)$ with SSNM connection admits a conformal $\eta$-Ricci-Yamabe soliton whose potential vector field $V$ is a HPCV field, then $\lambda=\frac{1}{2}[\b\Tilde{r}+(p+\frac{2}{2n+1})]-4n\a-\mu-\gamma$ and the soliton is shrinking if $\b\Tilde{r}<8n\a+2\mu+2\gamma-(p+\frac{2}{2n+1})$, steady if $\b\Tilde{r}=8n\a+2\mu+2\gamma-(p+\frac{2}{2n+1})$ or, expanding if $\b\Tilde{r}>8n\a+2\mu+2\gamma-(p+\frac{2}{2n+1})$.
\end{corollary}

 \section{$\Tilde{\tau}$-Curvature with SSNM connection:}\label{sec4}
In an $n$-dimensional Reimannian manifold $(M,g)$, a $\tau$-curvature tensor is a tensor of type $(1,3)$ introduced by M. M. Tripathi and P. Gupta \cite{MR2808047} in 2011 as follows: 
\begin{align}\label{eq29}
\tau(X,Y)Z=A_0R(X,Y)Z+A_1S(Y,Z)X+A_2S(X,Z)Y+A_3S(X,Y)Z+A_4g(Y,Z)QX\nonumber\\+A_5g(X,Z)QY+A_6g(X,Y)QZ+A_7r(g(Y,Z)X-g(X,Z)Y),
\end{align}
where $A_0,.......,A_7$ are smooth functions on $M$ and $R,~S$ and $r$ are defined as Riemannian curvature tensor, Ricci curvature tensor and scalar curvature respectively, with respect to the Levi civita connection $\nabla$. We define $\tau$-curvature tensor $\Tilde{\tau}$ with respect to the SSNM connection $\Tilde{\nabla}$ as follows:
\begin{align}\label{eq30}
    \Tilde{\tau}(X,Y)Z=A_0\Tilde{R}(X,Y)Z+A_1\Tilde{S}(Y,Z)X+A_2\Tilde{S}(X,Z)Y+A_3\Tilde{S}(X,Y)Z+A_4g(Y,Z)\Tilde{Q}X\nonumber\\+A_5g(X,Z)\Tilde{Q}Y+A_6g(X,Y)\Tilde{Q}Z+A_7\Tilde{r}(g(Y,Z)X-g(X,Z)Y),
\end{align}
where $\Tilde{r},~\Tilde{R}$ and $\Tilde{S}$ represent the scalar curvature, the Riemannian curvature tensor and Ricci curvature tensor on $M$ with SSNM connection $\Tilde{\nabla}$ respectively.
\begin{theorem}
    In a $(2n+1)$ dimensional $K$-contact manifold $(M,g)$, the $\tau$-curvature tensor $\Tilde{\tau}$ with respect to SSNM connection $\Tilde{\nabla}$ and $\tau$ with Levi-Civita connection $\nabla$ are related by 
    \begin{eqnarray}\label{eq31}
                \Tilde{\tau}(X,Y,Z,W)=\tau(X,Y,Z,W)+(A_0+2nA_1)\{g(Z,\varphi Y)g(X,W)+\eta(Z)\eta(Y)g(X,W)\}\nonumber\\-(A_0-2nA_2)\{\eta(X)\eta(Z)g(Y,W)+g(\varphi X,Z)g(Y,W)\}+2nA_3\{\eta(Y)\eta(X)g(Z,W)\nonumber\\-g(X,\varphi Y)g(Z,W)\}+2nA_4\{\eta(W)\eta(X)g(Z,Y)-g(X,\varphi W)g(Z,Y)\}\nonumber\\+2nA_5\{\eta(W)\eta(Y)g(Z,X)-g(Y,\varphi W)g(Z,X)\}+2nA_6\{\eta(W)\eta(Z)g(X,Y)\nonumber\\-g(Z,\varphi W)g(X,Y)\}.
    \end{eqnarray}
\end{theorem}
\begin{proof}
    Let $M$ be a $(2n+1)$-dimensional $K$-contact manifold equipped with a Riemannian metric $g$. From (\ref{eq30}), we have:
    \begin{align*}
        \Tilde{\tau}(X,Y,Z,W)=A_0\Tilde{R}(X,Y,Z,W)+A_1\Tilde{S}(Y,Z)g(X,W)+A_2\Tilde{S}(X,Z)g(Y,W)+\\A_3\Tilde{S}(X,Y)g(Z,W)+A_4g(Y,Z)g(\Tilde{Q}X,W)+A_5g(X,Z)g(\Tilde{Q}Y,W)\\+A_6g(X,Y)g(\Tilde{Q}Z,W)+A_7\Tilde{r}(g(Y,Z)g(X,W)-g(X,Z)g(Y,W)).
    \end{align*}
    Using (\ref{eq14}) and (\ref{eq17}) in above equation, we obtain
    \begin{eqnarray*}
                \Tilde{\tau}(X,Y,Z,W)=A_0R(X,Y,Z,W)+A_1S(Y,Z)g(X,W)+A_2S(X,Z)g(Y,W)\nonumber\\+A_3S(X,Y)g(Z,W)+A_4g(Y,Z)g(QX,W)+A_5g(X,Z)g(QY,W)\nonumber\\+A_6g(X,Y)g(QZ,W)+A_7(r+2n)(g(Y,Z)g(X,W)-g(X,Z)g(Y,W))\nonumber\\+A_0\{g(Z,\varphi Y)g(X,W)-g(Z,\varphi X)g(Y,W)-\eta(Z)\eta(X)g(Y,W)\nonumber\\+\eta(Z)\eta(Y)g(X,W)\}+2nA_1\{\eta(Y)\eta(Z)g(X,W)-g(Y,\varphi Z)g(X,W)\}\nonumber\\+2nA_2\{\eta(X)\eta(Z)g(Y,W)-g(X,\varphi Z)g(Y,W)\}+2nA_3\{\eta(Y)\eta(X)g(Z,W)\nonumber\\-g(X,\varphi Y)g(Z,W)\}
                +2nA_4\{\eta(W)\eta(X)g(Z,Y)-g(X,\varphi W)g(Z,Y)\}\nonumber\\+2nA_5\{\eta(W)\eta(Y)g(Z,X)-g(Y,\varphi W)g(Z,X)\}\nonumber\\+2nA_6\{\eta(W)\eta(Z)g(X,Y)-g(Z,\varphi W)g(X,Y)\}.
    \end{eqnarray*}
We simplify the above equation by using (\ref{eq29}) and we get (\ref{eq31}).
\end{proof}
Now we compute several $\Tilde{\tau}$-curvature tensor identities that will be utilized in subsequent results.
\begin{align*}
    1.~~\Tilde{\tau}(X,Y)\xi=A_0\Tilde{R}(X,Y)\xi+A_1\Tilde{S}(Y,\xi)X+A_2\Tilde{S}(X,\xi)Y+A_3\Tilde{S}(X,Y)\xi+A_4g(Y,\xi)\Tilde{Q}X\\+A_5g(X,\xi)\Tilde{Q}Y+A_6g(X,Y)\Tilde{Q}\xi+A_7\Tilde{r}(g(Y,\xi)X-g(X,\xi)Y).
\end{align*}
Using $(\ref{eq16})$ and $(\ref{eq18})$ in above equation, we get
\begin{align}\label{eq32}
    =(2A_0+4nA_1+A_7\Tilde{r})\eta(Y)X-(2A_0-4nA_2+A_7\Tilde{r})\eta(X)Y+4nA_6g(X,Y)\xi~~~~~~~~\nonumber\\+A_3\Tilde{S}(X,Y)\xi+A_4\eta(Y)\Tilde{Q}X
    +A_5\eta(X)\Tilde{Q}Y.~~~~~~~~~~~~~~~~~~~~~~~~~~~~~~~~~~~~~
\end{align}
\begin{align*}
    2.~~\Tilde{\tau}(\xi,Y)Z=A_0\Tilde{R}(\xi,Y)Z+A_1\Tilde{S}(Y,Z)\xi+A_2\Tilde{S}(\xi,Z)Y+A_3\Tilde{S}(\xi,Y)Z+A_4g(Y,Z)\Tilde{Q}\xi\\+A_5g(\xi,Z)\Tilde{Q}Y
+A_6g(\xi,Y)\Tilde{Q}Z+A_7\Tilde{r}(g(Y,Z)\xi-g(\xi,Z)Y).
\end{align*}
Using $(\ref{eq16})$ and $(\ref{eq18})$ in above equation, we get:
\begin{align*}
    =(A_0+4nA_4+A_7\Tilde{r})g(Y,Z)\xi-(2A_0-4nA_2+A_7\Tilde{r})\eta(Z)Y+A_0(g
(\varphi Y,Z)\xi+\eta(Z)\eta(Y)\xi)\\+A_1\Tilde{S}(Y,Z)\xi+4nA_3\eta(Y)Z+A_5\eta(Z)\Tilde{Q}Y+A_6\eta(Y)\Tilde{Q}Z
.~~~~~~~~~~~~~~~~~~~~~~~~~~~~~~~~~~~~~~~~~
\end{align*}
Again replacing $Z$ with $\xi$, we obtain:
\begin{align}\label{eq33}
    \Tilde{\tau}(\xi,Y)\xi=(2A_0+4n(A_1+A_3+A_4+A_6)+\Tilde{r}A_7)\eta(Y)\xi-(2A_0-4nA_2+\Tilde{r}A_7)Y+A_5\Tilde{Q}Y.
\end{align}
Suppose that a K-contact manifold with SSNM connection is $\Tilde{\tau}$ semi-symmetric, that is, $(\Tilde{R}(X,Y)\cdot\Tilde{\tau})(U,V)W=0$, This yields:
\begin{equation}\label{eq34}
    \Tilde{R}(X,Y)\Tilde{\tau}(U,V)W-\Tilde{\tau}(\Tilde{R}(X,Y)U,V)W-\Tilde{\tau}(U,R(X,Y)V)W-\Tilde{\tau}(U,V)\Tilde{R}(X,Y)W=0,
\end{equation}
setting $X=U=W=\xi$, we have,
\begin{equation}\label{eq35}
    \Tilde{R}(\xi,Y)\Tilde{\tau}(\xi,V)\xi-\Tilde{\tau}(\Tilde{R}(\xi,Y)\xi,V)\xi-\Tilde{\tau}(\xi,R(\xi,Y)V)\xi-\Tilde{\tau}(\xi,V)\Tilde{R}(\xi,Y)\xi=0.
\end{equation}
Using $(\ref{eq16}),~(\ref{eq17})$ and $(\ref{eq33})$, we obtain:
\begin{align}\label{eq36}
   \Tilde{R}(\xi,Y)\Tilde{\tau}(\xi,V)\xi =[2(2A_0+4n(A_1+A_3+A_4+A_6)+\Tilde{r}A_7)-(2A_0-4nA_2+\Tilde{r}A_7)]\eta(V)\eta(Y)\xi\nonumber\\-[2(2A_0+4n(A_1+A_3+A_4+A_6)+\Tilde{r}A_7)-2(2A_0-4nA_2+\Tilde{r}A_7)]\eta(V)Y\nonumber\\-(2A_0-4nA_2+\Tilde{r}A_7)\{g(Y,V)\xi+g(\varphi Y,V)\xi\}+A_5\{g(Y,\Tilde{Q}V)\xi+g(\varphi Y,\Tilde{Q}V)\xi\nonumber\\-2\eta(\Tilde{Q}V)Y+\eta(\Tilde{Q}V)\eta(Y)\xi\}.
\end{align}
Now using (\ref{eq15}), and (\ref{eq32}), we get:  
\begin{align}\label{eq37}
    \Tilde{\tau}(\Tilde{R}(\xi,Y)\xi,V)\xi=2(2A_0+4n(A_1+A_3+A_4+A_6)+\Tilde{r}A_7)\eta(Y)\eta(V)\xi\nonumber\\-2(2A_0+4nA_1+A_7\Tilde{r})\eta(V)Y-8nA_6g(Y,V)\xi-2A_3\Tilde{S}(Y,V)\xi-2A_4\eta(V)\Tilde{Q}Y.
\end{align}
Also, we have:
\begin{align}\label{eq38}
    \Tilde{\tau}(\xi,\Tilde{R}(\xi,Y)V)\xi={\scriptsize
    4n(A_1+A_2+A_3+A_4+A_5+A_6)}\{g(Y,V)\xi+g(\varphi Y,V)\xi\}~~~~~~~~~~~~~~~~~~~\nonumber\\+{\scriptsize[4n(A_1+A_2+A_3+A_4+A_5+A_6)-2(2A_0+4n(A_1+A_3+A_4+A_6)+\Tilde{r}A_7)]}\eta(Y)\eta(V)\xi\nonumber\\+2(2A_0-4nA_2+\Tilde{r}A_7)\eta(V)Y-2A_5\eta(V)\Tilde{Q}Y,~~~~~~~~~~~~~~~~~~~~~~~~~~~~~~~~
\end{align}
and
\begin{align}\label{eq39}
    \Tilde{\tau}(\xi,V)\Tilde{R}(\xi,Y)\xi=2(A_0+4n(A_1+A_3+A_4+A_6)+\Tilde{r}A_7)\eta(Y)\eta(V)\xi\nonumber\\-2(A_0+4nA_4+A_7\Tilde{r})g(V,Y)\xi-2A_0g
(\varphi V,Y)\xi-2A_1\Tilde{S}(V,Y)\xi-8nA_3\eta(V)Y-2A_6\eta(V)\Tilde{Q}Y.
\end{align}
Putting the values of (\ref{eq36}), (\ref{eq37}), (\ref{eq38}) and (\ref{eq39}) in (\ref{eq35}), we get
\begin{align}\label{eq40}
    [-\Tilde{r}A_7-4n(A_1+A_3+A_4+A_5+A_6)]\eta(Y)\eta(V)\xi-[4n(A_1+A_3-A_4+A_5-A_6)\nonumber\\-\Tilde{r}A_7]g(V,Y)\xi-[8n(A_4+A_6)]\eta(V)Y-[4A_0+\Tilde{r}A_7+4n(A_1+\nonumber\\A_3+A_4+A_5+A_6)]g
(\varphi Y,V)\xi+A_5\{g(Y,\Tilde{Q}V)\xi+g(\varphi Y,\Tilde{Q}V)\xi\nonumber\\-2\eta(\Tilde{Q}V)Y+\eta(\Tilde{Q}V)\eta(Y)\xi\}+2A_3\Tilde{S}(Y,V)\xi+2(A_4+A_5\nonumber\\+A_6)\eta(V)\Tilde{Q}Y+2A_1\Tilde{S}(V,Y)\xi=0.
\end{align}
If the Ricci tensor S of the manifold $M^{2n+1}$ is parallel, $M$ becomes an Einstein manifold. Thus, from $(\ref{eq17})$, we get:
\begin{equation}\label{eq41}
    \Tilde{S}(X,Y)=2ng(X,Y)-2ng(X,\varphi Y)+2n\eta(X)\eta(Y)
\end{equation}

\begin{equation}\label{eq42}
    \Tilde{Q}X=2n(X+\varphi X+\eta(X)\xi)
\end{equation}
Using $(\ref{eq41})$ and $(\ref{eq42})$ in $(\ref{eq40})$, we get
\begin{align*}
     [4n(A_4+A_6)+\Tilde{r}A_7]g(V,Y)\xi-[4n(A_4+A_5+A_6)](\eta(V)Y-\eta(V)\varphi Y)\\+[4nA_5-\Tilde{r}A_7]\eta(Y)\eta(V)\xi-[4A_0+\Tilde{r}A_7+4n(2A_1+A_4+A_5+A_6)]g
(\varphi Y,V)\xi=0
\end{align*}
So overview the above equation, we can state the following theorem:
\begin{theorem}
    A $K$-contact manifold $(M,g)$ with SSNM connection is $\Tilde{\tau}$-semi-symmetric if it satisfies the conditions $A_4+A_5+A_6=0,~\Tilde{r}A_7=4nA_5$ and $4A_0+8nA_1+\Tilde{r}A_7=0$.
\end{theorem}
When the $\Tilde{\tau}-$curvature tensor reduces to a quasi-conformal curvature tensor if it satisfies the conditions $A_1=-A_2=A_4=-A_5,~ A_3=A_6=0~ and~A_7=-\frac{1}{2n+1}(\frac{A_0}{2n}+2A_1).$
Then we can state the following theorem:
\begin{theorem}
    A $K$-contact manifold $(M,g)$ with SSNM connection is quasi-conformally-semi-symmetric if it satisfies the condition $A_0+nA_1=0.$
\end{theorem}
Furthermore, if the $\Tilde{\tau}-$curvature tensor reduces to a pseudo-projective curvature tensor i.e, $A_1=A_2,~A_3=A_4=A_5=A_6=0,~and~A_7=-\frac{1}{2n+1}(\frac{A_0}{2n}+A_1)$, then we obtain
\begin{theorem}
    A $K$-contact manifold $(M,g)$ with SSNM connection is pseudo-projective-semi-symmetric if satisfies the condition $A_0+2nA_1=0.$
\end{theorem}

{\bf Acknowledgement.} The author Bidhan Mondal thanks the University Grants Commission Junior Research Fellowship (Id No. 231610077944) for their financial assistance.
\medskip\\
\textbf{Author Contributions:} The authors contributed equally to this work.
\medskip\\
\textbf{Data Availability Statements:} Data sharing not applicable.
\section{Declarations}
\textbf{Conflict of interest:} The authors declare that they have no conflict of interest nor competing interests.

\bibliographystyle{amsplain}

\end{document}